\documentclass[11pt]{amsart}

\usepackage{cmap}
\usepackage[T2A]{fontenc}
\usepackage[utf8]{inputenc}
\usepackage[english]{babel}
\usepackage{amsfonts,amssymb}
\usepackage{hyperref}
\usepackage{stmaryrd}
\usepackage[left=2.0cm,right=2.0cm,top=2.0cm,bottom=2.0cm]{geometry}

\usepackage[all,cmtip]{xy}
\usepackage{capt-of} 

\newcommand{\mytitle}{The reverse decomposition of unipotents for bivectors}
\title{\mytitle}

\author{Roman~Lubkov}
\thanks{This publication is supported by RFBR, project number 19-31-90072}
\address{St.~Petersburg Department of V.\,A.\,Steklov Institute of Mathematics of the Russian Academy of Sciences}
\email{RomanLubkov@yandex.ru}

\keywords{general linear group, elementary group, fundamental representation, decomposition of unipotents, reverse decomposition of unipotents, sandwich classification}

\subjclass{20G35}
\date{}
\DeclareMathOperator{\E}{E}
\DeclareMathOperator{\A}{A}
\DeclareMathOperator{\D}{D}

\DeclareMathOperator{\CC}{C}
\DeclareMathOperator{\SL}{SL}
\DeclareMathOperator{\GL}{GL}
\DeclareMathOperator{\sign}{sgn}
\DeclareMathOperator{\Plu}{Pl\ddot{u}}
\DeclareMathOperator{\Stabi}{Stab}
\DeclareMathOperator{\height}{ht}
\DeclareMathOperator{\level}{lev}

\renewcommand{\trianglelefteq}{\trianglelefteqslant}
\renewcommand{\leq}{\leqslant}
\renewcommand{\geq}{\geqslant}

\newcommand{\bw}[1]{\mathord{\raisebox{2pt}
{\hbox{$\scriptstyle{\bigwedge^{\!#1}}$}}}}

\theoremstyle{plain}
\newtheorem{theorem}{Theorem} 
\newtheorem{prop}{Proposition}
\newtheorem{lemma}{Lemma}

\theoremstyle{remark}
\newtheorem*{remark}{Remark}

\theoremstyle{definition}
\newtheorem*{example}{Example}

\begin{document}

\begin{abstract}
For the second fundamental representation of the general linear group over a commutative ring $R$ we construct straightforward and uniform polynomial expressions of elementary generators as products of elementary conjugates of an arbitrary matrix and its inverse. Towards the solution we get stabilization theorems for any column of a matrix from $\GL_{\binom{n}{2}}(R)$ or from the exterior square of $\GL_n(R)$, $n\geq 3$.
\end{abstract}

\maketitle

\section*{Introduction}\label{intro}

Alexei Stepanov proposed the decomposition of unipotens for $\GL_n(R)$ in 1987~\cite{StepPhD}. Almost at once Nikolai Vavilov generalized this result to other split classical groups~\cite{VavHabilit}, and in 1990 Nikolai Vavilov, Alexei Stepanov, and Eugene Plotkin developed the method for exceptional groups~\cite{VPSComput}. Since the 1990s the decomposition of unipotents was the focus of a number of authors, see~\cite{PetDecomp,StepDecomp,StepVavDecomp,VavDecomp25,VavKaz} for further references.

Let $\Phi$ be an irreducible root system of rank $\geq 2$, $R$ be an arbitrary commutative ring with $1$, and $G(\Phi,R)$ be the simply connected Chevalley group of type $\Phi$ over $R$. Fix a split maximal torus $T(\Phi,R)$ of $G(\Phi,R)$ and the corresponding elementary generators $x_{\alpha}(\xi)$,
where $\alpha\in\Phi$, $\xi\in R$. Let $E(\Phi,R)$ be the elementary subgroup spanned by all these elementary generators.

The decomposition of unipotents can be viewed as an effective version of the normality of elementary subgroups, i.\,e., that $E(\Phi,R)$ is normal in $G(\Phi,R)$. In particular, for simply laced
root systems admitting microweights the method provides straightforward formulae for $gx_{\alpha}(\xi)g^{-1}$ as a product of elementary matrices, where $g$ is an arbitrary matrix in $G(\Phi,R)$.
Let $gx_{\alpha}(\xi)g^{-1}$ be a product of $\leq L$ elementary generators. Then $L$ equals $4(l+1)l$ for $\A_l$, $4\cdot 2l\cdot 2(l-1)$ for $\D_l$, $4\cdot 27\cdot 16$ for $\E_6$, and $4\cdot 56\cdot 27$ for $\E_7$, see~\cite{HSVZwidth,VPSComput,VavThird}.

The reverse decomposition of unipotents was proposed by Raimund Preusser in 2017. For classical groups the method provides similar explicit short polynomial expressions of the elementary generators of $g^{E(\Phi,R)}$ in terms of the elementary conjugates of $g$ and $g^{-1}$, see~\cite{PreusserHyp2,PreusserOUodd,PreusserGLOUeven}. Soon Nikolai Vavilov and Zuhong Zhang generalized the idea to exceptional groups and developed it for relative elementary matrices $x\in \E_n(J)$ or $x\in \E_n(R,J)$, where $J$ is an ideal of $R$~\cite{VavTowards,VavTowRel}.

In the present paper, we produce a variation of the reverse decomposition of unipotents for the exterior square of an elementary group. For a commutative ring $R$ by $\bw{2} R^{n}$ we denote the second exterior power of the free module $R^n$. We consider the natural transformation,
$$\bw{2}: \GL_{n} \rightarrow \GL_{\binom{n}{2}}$$
which extends the action of the group $\GL_{n}(R)$ on the module $\bw{2} R^{n}$. 

An elementary group $\E_n(R)$ is a subgroup of the group of points $\GL_{n}(R)$, so its exterior square $\bw{2}\E_n(R)$ is a well-defined subgroup of the group of points $\bw{2}\GL_{n} (R)$. Recall that the Pl\"ucker relations consist of vanishing certain homogeneous quadratic polynomials $f_{I,J}\in R[x_H]_{H\in \bw{2}[n]}$ of Grassmann coordinates $x_H$. For $m$--exterior power there are short and long Pl\"ucker relations, but for $m=2$ all relations are short and have the following form
$$f_{I,J}=\sum\limits_{j\in J\backslash I}\pm x_{I\cup \{j\}}x_{J\backslash \{j\}},$$
where $I\in\bw{1}[n]$ and $J\in\bw{3}[n]$, see~\cite{VavPere}. The second exterior power of $\GL_n(R)$ [for $n\neq 4, n\geq 3$] is a stabilizer of the Pl\"ucker ideal
$$\Plu(n,R)\trianglelefteq R\,[x_{I}:\, I \in \bw{2}\{1,\dots,n\}],$$
where $\Plu(n,R)$ is generated by the standard short Pl\"ucker relations.

Further, the \textit{upper level} of a matrix $g\in\bw{2}\GL_n(R)$ is the smallest ideal $A=\level(g)\trianglelefteq R$ such that $g$ belongs to the full preimage of the center of $\bw{2}\GL_{n}(R/A)$ under the reduction homomorphism $\bw{2}\GL_{n}(R)\longrightarrow \bw{2}\GL_{n}(R/A)$. The upper level is generated by the $\binom{n}{2}^2-\binom{n}{2}$ off-diagonal entries and by the $\binom{n}{2}-1$ pair-wise differences of its diagonal entries. Overall we have $\binom{n}{2}^2-1$ elements. 

Notice that the group $\bw{2}\GL_n(R)$ has an invariant form $f$ for even $n$ and an ideal of forms $K$ for odd $n$. Thus, $\bw{2}\GL_n(R)$ can be regarded as a stabilizer of $f$ or $K$ in the sense of semi-invariance, see~\cite{LubNekoverExPowArx}. We stress that $\bw{2}\big(\GL_{n}(R)\big)$ do not equal $\bw{2}\GL_n(R)$. The first group is the image of the general linear group under the Cauchy--Binet homomorphism, whereas the second one is the group of $R$-points of the image of the group scheme $\GL_{n}$ under the natural transformation. We refer the reader to~\cite{VavPere}, where this point is discussed in detail.

Note also that the exterior square of the general linear group for $n=3$ is isomorphic to $\GL_3(R)$. This case has been completely solved by Raimund Preusser in~\cite{PreusserGLOUeven}. So we work only in the case $n\geq 4$.

In these terms the reverse decomposition of unipotens for the second exterior powers looks as follows.
\begin{theorem}
Let R be a commutative ring, $n\geq 4$, and $g\in\bw{2}\GL_n(R)$. Then for any $\xi\in \level(g)$, $1\leq k\neq l\leq n$, the transvection $\bw{2}t_{k,l}(\xi)$ is a product of $\leq 8\big(\binom{n}{2}^2-1\big)$ elementary conjugates of $g$ and $g^{-1}$.
\end{theorem}

A key step in the proofs of all such results is the stabilization theorems for a column of an arbitrary matrix $g\in G(\Phi,R)$. Therefore, we have to show that an analogous result is true in the case of $\bw{2}\GL_n(R)$. The following result is not astounding, because many experts believed in this fact in the early 2000s.
\begin{theorem}
Let $w$ be any column of a matrix in $\bw{2}\GL_n(R)$, $n\geq 5$. Set $T_1:=\bw{2}t_{2,3}(w_{45})\bw{2}t_{2,4}(-w_{35})\bw{2}t_{2,5}(w_{34})$. Then $T_1\cdot w=w$.
\end{theorem}
We could use this result for the deal, but the rank $n$ should be at least $5$. Besides, the vector $w$ is not arbitrary --- we require that the Pl\"ucker relations hold. So we have a natural question: is there a transvection $T$ in $\bw{2}\E_n(R)$ such that $T$ stabilizes an arbitrary column of a matrix in $\GL_{\binom{n}{2}}(R))$? The following result gives a positive answer.

\begin{theorem}
Let $w$ be an arbitrary vector in $R^{\binom{n}{2}}$, $n\geq 3$. Set $T_{*,j}:=\prod\limits_{s\neq j}\bw{2}t_{s,j}(\sign(s,j)w_{sj})$, $j\in\{1,\dots, n\}$, $\sign(s,j)$ are the structure constants for $\bw{2}\E_n(R)$. Then $T_{*,j}\cdot w=w$.
\end{theorem}

Let us emphasize the hidden features of the theorem. In addition to the proof of the reverse decomposition of unipotents with a smaller constraint we can use it for other structure problems, e.\,g., for the standard description of overgroups. Ilia Nekrasov and the author also work in this direction to complete a solution of this general problem, see~\cite{LubNekoverExPowArx}.

The present paper is organized as follows. In Section~\ref{princnot} we recall the basic definitions pertaining to elementary groups and the second fundamental representation. Stabilization theorems~\ref{stabil2} and~\ref{stabil} are discussed in Section~\ref{stab}. The core of the paper is Section~\ref{RDU}, which is devoted directly to the proof of Theorem~\ref{RDT}.

\textbf{Acknowledgment.} I am very grateful to my teachers Nikolai Vavilov and Alexei Stepanov for their constant support during the writing of the present paper. I also thank Ilia Nekrasov for thoughtful reading of the original manuscript and  the referee for useful notes in a preliminary version of this paper.

\section{Notation}\label{princnot}
One can find many details relating to Chevalley groups over rings and further references in~\cite{StepVavDecomp}. Here we only fix the main notation.

First, let $G$ be a group. By a commutator of two elements we always mean the \textit{left-normed} commutator $[x,y]=xyx^{-1}y^{-1}$, where $x,y\in G$. By ${}^xy=xyx^{-1}$ and $y^x=x^{-1}yx$ we denote the left and the right conjugates of $y$ by $x$, respectively. In the sequel, we will also use the notation $y^{\pm x}$ for the the right conjugates of $y$ by $x$ or $x^{-1}$.

Let $R$ be an associative ring with 1. By default, it is assumed to be commutative. By $\GL_n(R)$ we denote the general linear group, while $\SL_{n}(R)$ is the special linear group of degree $n$ over $R$. Let $g=(g_{i,j})$, $1\leq i,j\leq n$ be an invertible matrix in $\GL_n(R)$. Entries of the inverse matrix $g^{-1}$ are denoted by $g'_{i,j}$, where $1\leq i,j\leq n$.

By $t_{i,j}(\xi)$ we denote an elementary transvection, i.\,e., a matrix of the form $t_{i,j}(\xi)=e+\xi e_{i,j}$, $1\leq i\neq j\leq n$, $\xi\in R$. Here $e$ denotes the identity matrix and $e_{i,j}$ denotes a standard matrix unit, i.\,e., a matrix that has 1 at the position $(i,j)$ and zeros elsewhere. Hereinafter, we use (without any references) the standard relations among elementary transvections such as
\begin{enumerate}
\item additivity:
$$t_{i,j}(\xi)t_{i,j}(\zeta)=t_{i,j}(\xi+\zeta).$$
\item the Chevalley commutator formula:
$$[t_{i,j}(\xi),t_{h,k}(\zeta)]=
\begin{cases}
e,& \text{ if } j\neq h, i\neq k,\\
t_{i,k}(\xi\zeta),& \text{ if } j=h, i\neq k,\\
t_{h,j}(-\zeta\xi),& \text{ if } j\neq h, i=k.
\end{cases}$$
\end{enumerate}

By $R^n$ we denote the free $R$--module. It contains columns with components in $R$. Vectors $e_1,\dots,e_n$ denote the standard basis of $R^n$. Let $P_m$ be a parabolic subgroup of the coordinate subspace $\langle e_1,\dots,e_m\rangle$. It equals the stabilizer $\Stabi(\langle e_1,\dots,e_m\rangle)$. Further, let $U_m$ be the subgroup of $P_m$, generated by $t_{i,j}(\xi)$, where $1\leq i\leq m$, $m+1\leq j\leq n$, $\xi\in R$. It is called the uniponet radical of $P_m$. Clearly, $U_m$ is normal and abelian.

Now, let $I$ be an ideal in $R$. Denote by $\E_n(I)$ the subgroup of $\GL_{n}(R)$ generated by all elementary transvections of level $I$: $t_{i,j}(\xi), 1\leq i\neq j\leq n, \xi\in I$. In the primary case $I = R$, the group $\E_n(R)$ is called the [absolute] elementary group. It is well known (due to Andrei Suslin~\cite{SuslinSerreConj}) that the elementary group is normal in the general linear group $\GL_{n}(R)$ for $n \geq 3$. Further, the relative elementary subgroup $\E_n(R,I)$ of level $I$ is defined as the normal closure of $\E_n(I)$ in the absolute elementary subgroup $\E_n(R)$:
$$\E_n(R,I)=\langle t_{i,j}(\xi), 1\leq i\neq j\leq n, \xi\in I\rangle^{\E_n(R)}.$$

Let $[n]$ be the set $\{1,2,\ldots, n\}$ and let $I=\{i_1,i_2\},i_1\neq i_2$ be a subset of $[n]$. Define an exterior square of $[n]$. Elements of this set are ordered subsets $I\subseteq [n]$ of cardinality $2$ without repeating entries:
$$\bw{2}[n] := \{ (i_{1}, i_{2})\; |\; i_{j} \in [n], i_{1} \neq i_{2} \}.$$ 
Usually we record $i_j$ in ascending order, $i_1<i_2$. Let $\sign(i_1,i_2):=+1$, if $i_1<i_2$ and $\sign(i_1,i_2):=-1$, if $i_1>i_2$.

Let $I, J$ be two elements of $\bw{2}[n]$. We define a \textit{height} of  the  pair $(I,J)$ as the cardinality of the intersection $I\cap J$: $\height(I,J):=|I\cap J|$. This combinatorial characteristic plays the same role as the distance function $d(\lambda, \mu)$ for roots $\lambda$ and $\mu$ in a weight graph in root system terms.

Let $R$ be a commutative ring, $n\geq 3$, and $N$ be the binomial coefficient $\binom{n}{2}$. We consider the standard action of the group $\GL_{n}(R)$ on the $R$--module $R^n$. Let us define the exterior square $\bw{2}(R^n)$ of the free $R$--module $R^n$ as follows.
The standard basis of $\bw{2}(R^n)$ consists of exterior products $e_i\wedge e_j$, $1~\leq~i~<~j~\leq~n$. However, we also define $e_i\wedge e_j$ for arbitrary pair $i,j$ with $e_i\wedge e_j=-e_j\wedge e_i$. The standard action of the group $\GL_{n}(R)$ on the module $\bw{2}(R^n)$ is defined on the basis elements as follows
$$g(e_i\wedge e_j):=(ge_i)\wedge(ge_j) \text{ for any } g\in \GL_{n}(R) \text{ and } 1\leq i\neq j\leq n.$$
This action is extended by linearity to the whole module $\bw{2}(R^n)$. We define a subgroup $\bw{2}\big(\GL_{n}(R)\big)$ of the general linear group $\GL_{N}(R)$ as the image of $\GL_n(R)$ under the action morphism.

In other words, let us consider the Cauchy--Binet homomorphism
$$\bw{2}:\GL_{n}(R)\longrightarrow \GL_{N}(R),$$
mapping a matrix $x\in \GL_{n}(R)$ to the matrix $\bw{2}(x)\in \GL_{N}(R)$. Entries of $\bw{2}(x)\in \GL_{N}(R)$ are the second order minors of the matrix $x$. Then the group $\bw{2}\big(\GL_{n}(R)\big)$ is the image of the general linear group under the Cauchy--Binet homomorphism. It is natural to index entries of the matrix $\bw{2}(x)$ by a pair of elements of the set $\bw{2}[n]$:
$$\left(\bw{2}(x)\right)_{I,J}=\left(\bw{2}(x)\right)_{(i_1,i_2),(j_1,j_2)}:=M_{i_1,i_2}^{j_1,j_2}(x) = x_{i_1, j_1}\cdot x_{i_2, j_2} - x_{i_1, j_2}\cdot x_{i_2, j_1},$$
where $M_{I}^{J}(x)$ is the determinant of a submatrix formed by rows from the set $I$ and columns from the set $J$. 

Here and further we use the notation $t_{I,J}(\xi)$ for elementary transvections in the elementary group $\E_N(R)$. We write indices $I,J$ without any brackets in ascending order, e.\,g., $t_{12,13}(\xi)$ is a transvection with the entry $\xi$ at the intersection of the first row and the second column. Suppose $x\in \E_n(R)$; then the exterior square of $x$ can be presented as a product of elementary transvections $t_{I,J}(\xi)\in \E_N(R)$ with $I,J\in \bw{2}[n]$ and $\xi\in R$. The following result can be extracted from the very definition of $\bw{2}\big(\GL_{n}(R)\big)$, see~\cite[Proposition 2]{LubNekoverExPowArx}.

\begin{prop} \label{ImageOfTransvFor2}
Let $t_{i,j}(\xi)$ be an elementary transvection. For $n\geq 3$, matrix $\bw{2}t_{i,j}(\xi)$ can be presented as the following product:
$$\bw{2}t_{i,j}(\xi)=\prod\limits_{k=1}^{i-1} t_{ki,kj}(\xi)\,\cdot\prod\limits_{l=i+1}^{j-1}t_{il,lj}(-\xi)\,\cdot\prod\limits_{m=j+1}^n t_{im,jm}(\xi) \eqno(1)$$
for any $1\leq i<j\leq n$.
\end{prop}
As an example, $\bw{2}t_{1,3}(\xi)=t_{12,23}(-\xi)t_{14,34}(\xi)t_{15,35}(\xi)$ for $n=5$.

\section{Stabilization theorems}\label{stab}
Let $G=G(\Phi,R)$ be the simply connected Chevalley group of type $\Phi$ over a commutative ring $R$, where $\Phi$ is an irreducible root system of rank $\geq 2$. $E(\Phi,R)$ denotes the elementary subgroup spanned by all elementary generators $x_{\alpha}(\xi)$, $\alpha\in\Phi$, $\xi\in R$.

A crucial move to use the decomposition of unipotents for solving problems in various fields is a trick to stabilize a column [row] of an arbitrary matrix $g\in G(\Phi,R)$. For the general linear group, the orthogonal group, or other Chevalley groups in the natural representations this trick can be found in the paper~\cite{StepVavDecomp} by Alexei Stepanov and Nikolai Vavilov. Therefore, we present the similar maneuver for the general linear group in the bivector representation. In this section, we construct two special elements $T_{*,j}$ and $T_1$ in $\bw{2}\E_n(R)$ stabilizing an arbitrary column of a matrix $g\in\GL_N(R)$ or $g\in\bw{2}\GL_n(R)$ respectively.

\setcounter{theorem}{2}
\begin{theorem}\label{stabil}
Let $w$ be an arbitrary vector in $R^N$, $n\geq 3$. Set $T_{*,j}:=\prod\limits_{s\neq j}\bw{2}t_{s,j}(\sign(s,j)w_{sj})$, where $j$ is in $[n]$. Then $T_{*,j}\cdot w=w$.
\end{theorem} 

\begin{remark}
Similarly, the element $T_{i,*}:=\prod\limits_{s\neq i}\bw{2}t_{i,s}(\sign(i,s)z_{is})$, $i\in[n]$, stabilizes an arbitrary row $z\in{}^N\!R$, i.\,e., $z\cdot T_{i,*}=z$.
\end{remark}

To uncover {\it the idea of the proof} and {\it forlmulae appearing} in Thearem~\ref{stabil} the language of weight diagrams be useful. Here we recall the necessary portion of it. The exterior square of an elementary group $\bw{2}\E_n(R)$ corresponds to the representation with the highest weight $\varpi_2$ of Chevalley group of the type $A_{n-1}$. The weight diagram in this case is a half of the square with $(n-1)$ vertices, see~\cite[Figure 5]{atlas} and Figure~\ref{Fig1}. The action of the elementary group $\E_N(R)$ on any vector $w\in R^N$ can be presented on the diagram. An elementary transvection $t_{I,J}(\xi)$ is a root unipotent for $\E_N(R)$. So, $t_{I,J}(\xi)$-action is an addition of $J$th coordinate to $I$th coordinate of $w$ with the multiplier $\xi$: $t_{I,J}(\xi)w=w+\xi w_{J}e_{I}$. For $n=5$ the action $t_{12,13}(\xi)$ on the weight diagram looks as follows:

\centerline{
\def\objectstyle{\scriptscriptstyle}
\xymatrix @=2pc{
&&&*+[o][F-]{15}\ar@{-}[dr]^1 \ar@{-}[dl]_4\\
&&*+[o][F-]{14}\ar@{-}[dr] \ar@{-}[dl]_3&&*+[o][F-]{25}\ar@{-}[dl]\ar@{-}[dr]^2\\
*+[o][F-]{12}&*+[o][F-]{13}\ar@{-->}@/^1pc/[l]^-{\xi}\ar@{-}[l]_2\ar@{-}[dr]_1 &&*+[o][F-]{24}\ar@{-}[dl] \ar@{-}[dr]&&*+[o][F-]{35}\ar@{-}[dl]^4\ar@{-}[r]^3&*+[o][F-]{45}\\
&&*+[o][F-]{23}&&*+[o][F-]{34}}}

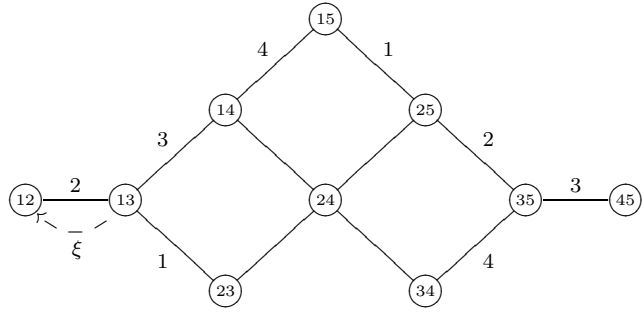
\captionof{figure}{The weight diagram $(A_4,\varpi_2)$ and the action of $t_{12,13}(\xi)$}\label{Fig1}

By formula~$(1)$ an exterior transvection $\bw{2}t_{i,j}(\xi)$ can be presented as a product of $(n-2)$ elementary transvections. Thus $T_{*,5}$ is a product of $(n-1)\cdot (n-2)$ elementary transvections in $\E_N(R)$. So we present the action of $T_{*,5}$ on the column $w\in R^N$ on this diagram. For the case $n=5$, $T_{*,5}$ is the following product of the root unipotens: $T_{*,5}=x_{\alpha_1+\alpha_2+\alpha_3+\alpha_4}(+w_{15})\cdot x_{\alpha_2+\alpha_3+\alpha_4}(+w_{25})\cdot x_{\alpha_3+\alpha_4}(+w_{35})\cdot x_{\alpha_4}(+w_{45})$. In terms of elementary transvections $T_{*,5}$ equals
\begin{multline*}
t_{12,25}(-w_{15})t_{13,35}(-w_{15})t_{14,45}(-w_{15})\cdot t_{12,15}(+w_{25})t_{23,35}(-w_{25})t_{24,45}(-w_{25})\\
\cdot t_{13,15}(+w_{35})t_{23,25}(+w_{35})t_{34,45}(-w_{35})\cdot t_{14,15}(+w_{45})t_{24,25}(+w_{45})t_{34,35}(+w_{45}).
\end{multline*}
We have shown on Figure~\ref{Fig2} all these transvections. Different types of arrows correspond to different uniponets. Note that the signs in the expression above appear due to the definition of exterior transvections. Namely indices $I=\{i_1,i_2\},J=\{j_1,j_2\}$ are not necessarily in ascending order, so we need to calculate $\sign(i_1i_2,j_1j_2)=\sign(i_1,i_2)\cdot\sign(j_1,j_2)$. Hence, $t_{I,J}(\xi)=t_{\sigma(I),\pi(J)}(\sign(\sigma)\sign(\pi)\xi)$ where $\sigma(I),\pi(J)$ are in ascending order.

\centerline{\xymatrix @=3pc{
&&&\bullet\ar@{.}[dr] \ar@{.}[dl]\ar@/_1pc/[dl]_(.7){w_{45}}\ar@{=>}@/_3pc/[ddll]_(.5){w_{35}}\ar@2{-->}@/_4pc/[ddlll]_(.6){w_{25}}\\
&&\bullet\ar@{.}[dr] \ar@{.}[dl]&&\bullet\ar@{.}[dl]\ar@{.}[dr]\ar@/_1pc/[dl]_(.7){w_{45}}\ar@{=>}@/_3pc/[ddll]_(.7){w_{35}}\ar@{-->}@/_3pc/[dllll]_(.2){-w_{15}}\\
\bullet&\bullet\ar@{.}[l]\ar@{.}[dr]&&\bullet\ar@{.}[dl] \ar@{.}[dr]&&\bullet\ar@{.}[dl]\ar@{.}[r]\ar@/_1pc/[dl]_(.7){w_{45}}\ar@2{-->}@/^5pc/[dlll]_(.8){-w_{25}}\ar@{-->}@/^8pc/[llll]^(.5){-w_{15}}&\bullet\ar@{=>}@/^1pc/[dll]^(.3){-w_{35}}\ar@2{-->}@/_1pc/[lll]_(.7){-w_{25}}\ar@{-->}@/_1pc/[ullll]_(.2){-w_{15}}\\
&&\bullet&&\bullet}}

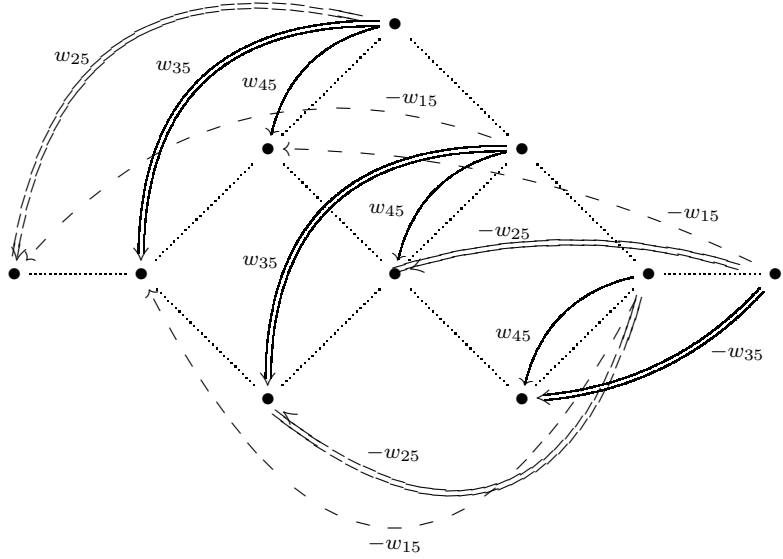
\captionof{figure}{The transvection $T_{*,5}$ on the weight diagram $(A_4,\varpi_2)$: each arrow stands for one elementary transvection $t_{I,J}(\xi)$, one type of the arrows corresponds to one root unipotent $x_{\alpha}(w)$}\label{Fig2}

\begin{proof}[Proof of Theorem $\ref{stabil}$]
The transvection $T_{*,j}$ acts on the vector $w$ by adding the expression
$$z(p,q,j):=\sign(pq,jq)\sign(p,j)w_{pj}w_{jq}+\sign(qp,jp)\sign(q,j)w_{qj}w_{jp}$$
to $\binom{n-1}{2}$ entries $w_{pq}$ of $w$, $pq\in\bw{2}([n]\setminus j)$, i.\,e., $(T_{*,j}w)_{pq}=w_{pq}+z(p,q,j)$. It is necessary to analyze $6$ cases for the numbers $1\leq p,q,j\leq n$. Below we give an analysis of these cases. By the above we must only prove that the expressions $z(p,q,j)$ vanish in all cases. They equal:
\begin{itemize}
    \item $(-1)(+1)w_{pj}w_{jq}+(+1)(+1)w_{qj}w_{jp}=0$ for $p<q<j$;
    \item $(+1)(+1)w_{pj}w_{jq}+(+1)(-1)w_{qj}w_{jp}=0$ for $p<j<q$;
    \item $(+1)(-1)w_{pj}w_{jq}+(-1)(-1)w_{qj}w_{jp}=0$ for $j<p<q$;
    \item $(-1)(-1)w_{pj}w_{jq}+(+1)(-1)w_{qj}w_{jp}=0$ for $j<q<p$;
    \item $(+1)(-1)w_{pj}w_{jq}+(+1)(+1)w_{qj}w_{jp}=0$ for $q<j<p$;
    \item $(+1)(+1)w_{pj}w_{jq}+(-1)(+1)w_{qj}w_{jp}=0$ for $q<p<j$.
\end{itemize}
\end{proof}

Now, let $g$ be a matrix in $\bw{2}\GL_n(R)$, i.\,e., columns $g_{*,J}$ satisfy the Pl\"ucker relations. Of course, $\bw{2}\GL_n(R)$ is a subgroup of $\GL_N(R)$, so we can use the elements $T_{*,j}$ to stabilize columns of $\bw{2}\GL_n(R)$. But $T_{*,j}$ contains $(n-1)$ elementary exterior transvections. For the case of $\bw{2}\GL_n(R)$ we construct the element in $\bw{2}\E_n(R)$ that contains \textit{three} elementary exterior transvections for an arbitrary rank $n$. 

The idea of the construction is the following. Choose one vertex on the weight diagram and choose exterior transvections so that the result of their action on this vertex is a short Pl\"ucker relation. At the same time the result on other vertices is either another Pl\"ucker relation or a trivial Pl\"ucker relation. Since $w$ is a column of a matrix in $\bw{2}\GL_n(R)$, we see that all Pl\"ucker relations equal zero. However, for this method $n$ should be at least $5$.

\setcounter{theorem}{1}
\begin{theorem}\label{stabil2}
Let $w$ be any column of a matrix in $\bw{2}\GL_n(R)$, $n\geq 5$. Set $T_1:=\bw{2}t_{2,3}(w_{45})\bw{2}t_{2,4}(-w_{35})\bw{2}t_{2,5}(w_{34})$. Then $T_1\cdot w=w$.
\end{theorem}
\begin{proof}
$T_1$ is a product of unipotens $x_{\alpha_2}(+w_{45})x_{\alpha_2+\alpha_3}(-w_{35})x_{\alpha_2+\alpha_3+\alpha_4}(+w_{34})$, see figure~\ref{Fig3}. This product acts on the vector $w$ by adding Pl\"ucker polynomials to $(n-1)$ entries of $w$. Namely, $(T_1w)_{2i}=w_{2i}+f_{i345}(w)$, where $i\in [n]\setminus 2$. Suppose that $i\in \{3,4,5\}$; then $f_{i345}(w)$ is a trivial polynomial $w_Aw_B-w_Bw_A=0$. Otherwise, $i\in [n]\setminus \{2,3,4,5\}$, $f_{i345}(w)=w_{45}w_{3i}-w_{35}w_{4i}+w_{34}w_{5i}$ is a short Pl\"ucker polynomial. Since entries of the vector $w$ satisfy the Pl\"ucker relations, we see that all polynomials $f_{i345}(w)$  equal zero.
\end{proof}

\centerline{\xymatrix @=2.0pc{
&&&&\bullet\ar@{-}[dr]\ar@{-}[dl]&\ar@{~}[ddddllll]\\
&&&\bullet\ar@{-}[dr] \ar@{-}[dl]\ar@{=>}@/_2pc/[ddlll]_(.3){w_{34}}&&\bullet\ar@{-}[dr]\ar@{-}[dl]\\
&&\bullet\ar@{-->}@/_1pc/[dll]_(.15){-w_{35}}\ar@{-}[dr] \ar@{-}[dl]&&\bullet\ar@{-}[dl]\ar@{-}[dr]&&\bullet\ar@{-}[dl]\ar@{-}[dr]\ar@/_1pc/[ul]_(.2){w_{45}}\\
\bullet&\bullet\ar@{-}[l]\ar@/_1pc/[l]_(.3){w_{45}}\ar@{-}[dr] &&\bullet\ar@{-}[dl] \ar@{-}[dr]&&\bullet\ar@{-}[dl]\ar@{-}[dr]\ar@/_1pc/[ul]_(.23){w_{45}}\ar@{=>}@/^5pc/[dlll]^(.5){w_{34}}&&\bullet\ar@{-}[dl]\ar@{-->}@/_2pc/[uull]_(.2){-w_{35}}\ar@{-}[r]&\bullet\ar@{=>}@/_3pc/[uulll]_(.5){w_{34}}\\
&&\bullet&&\bullet\ar@/_1pc/[ul]_(.3){w_{45}}\ar@{-->}@/^1pc/[ll]_(.5){-w_{35}}&&\bullet\ar@{=>}@/_1pc/[ulll]_(.8){w_{34}}\ar@{-->}@/_2pc/[uull]_(.3){-w_{35}}}}

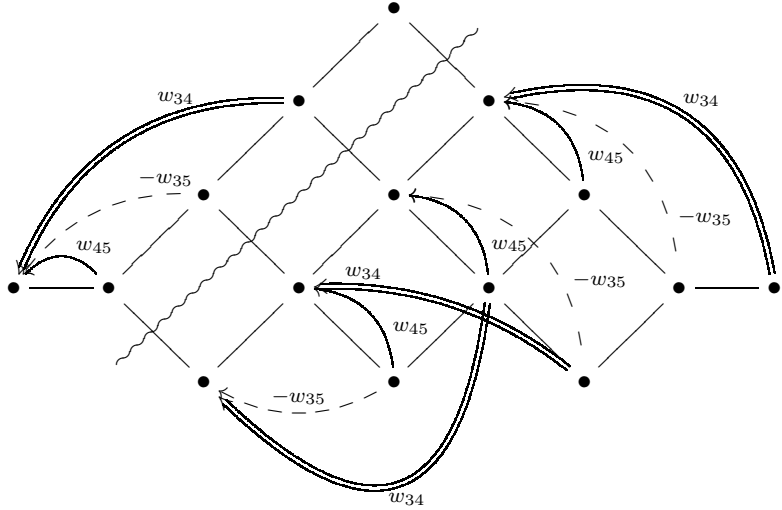
\captionof{figure}{The transvection $T_1$ on the weight diagram $(A_5,\varpi_2)$: the arrows of type ``$\rightarrow$'' correspond to $x_{\alpha_2}(w_{45})$, ``$\dashrightarrow$'' to $x_{\alpha_2+\alpha_3}(-w_{35})$, ``$\Rightarrow$'' to  $x_{\alpha_2+\alpha_3+\alpha_4}(w_{34})$}\label{Fig3}

Now we give a criterion that determine whether a matrix $g\in\GL_N(R)$ belongs to the group $\bw{2}\GL_n(R)$. Denote by $a_{A,C}^H(g)$ the sum $\sum\limits_{B\sqcup D = H} \sign(B,D) g_{B,A} g_{D,C}$, where $H\in\bw{4}[n]$, and $A,C\in\bw{2}[n]$. The next result can be found in~\cite[Theorem 3]{LubNek18}. It is an analog of~\cite[Proposition 4]{VP-Ep}, \cite[Proposition 1]{VP-EOodd}, and \cite[Theorem 5]{VavLuzgE6}.

\begin{prop}\label{criterion}
A matrix $g\in\GL_N(R)$ belongs to $\bw{2}\GL_n(R)$ if and only if the following equations hold:
\begin{itemize}
\item $a_{A,C}^H(g) = 0$ for any $H \in \bw{4}[n]$ and for any $A,C \in \bw{2}[n]$ such that $A\cap C\neq \emptyset$;
\item $\sign(A,C)\cdot a_{A,C}^H(g) = \sign(A',C')\cdot a_{A',C'}^H(g)$ for any $H \in \bw{4}[n]$ and for any $A,C,A',C'\in\bw{2}[n]$ such that $A\cup C=A'\cup B'$ and $A\cap C= \emptyset$.
\end{itemize}
\end{prop}
In terms of a distance on a weight graph the first type is an analog of \textit{equations on a pair of adjacent columns}, whereas the second type is an analog of \textit{equations on two pairs of nonadjacent columns}.

\begin{lemma}\label{equat}
Let $R$ be a commutative ring, $n\geq 3$, and  $g\in\bw{2}\GL_n(R)$. Suppose that one column of $g$ with index $I=\{i_1i_2\}$ is trivial; then for any indices $K\in\bw{2}([n]\setminus\{i_1,i_2\})$ and $J\in\bw{2}[n]$ such that $\height(I,J)=|I\cap J|=1$, we have $g_{K,J}=0$.
\end{lemma}
\begin{proof}
Suppose that a matrix $g$ belongs to $\bw{2}\GL_n(R)$; then by Proposition~\ref{criterion} for any indices $A,C\in\bw{2}[n],\, \height(A,C)=1$ and for any $H\in\bw{4}[n]$, we have
$$\sum\limits_{B\sqcup D = H} \sign(B,D) g_{B,A} g_{D,C}=0.$$
Let $g$ has the trivial column $A$. It follows that the above sum equals $g_{H\setminus A,C}$.

\end{proof}

\begin{example}
Let $g\in\bw{2}\GL_5(R)$ has the first trivial column, i.\,e., $g$ sits in the maximal parabolic subgroup $P_{12}$. Then by the latter lemma $g$ also sits in the submaximal parabolic subgroup ${}_{34}P_{12}$:
$$\left(\begin{smallmatrix}
1 & * & * & * & * & * & * & * & * & * \\
0 & * & * & * & * & * & * & * & * & * \\
0 & * & * & * & * & * & * & * & * & * \\
0 & * & * & * & * & * & * & * & * & * \\
0 & * & * & * & * & * & * & * & * & * \\
0 & * & * & * & * & * & * & * & * & * \\
0 & * & * & * & * & * & * & * & * & * \\ 
0 & 0 & 0 & 0 & 0 & 0 & 0 & * & * & * \\
0 & 0 & 0 & 0 & 0 & 0 & 0 & * & * & * \\
0 & 0 & 0 & 0 & 0 & 0 & 0 & * & * & *
\end{smallmatrix}\right)$$
\end{example}

\section{Reverse decomposition of unipotens}\label{RDU}

In this section, we explain the main idea of the reverse decomposition of transvections for the bivector representation of $\GL$. For this, we have to recall some notation.

Let $A\trianglelefteq R$ be an ideal, and let $R/A$ be the factor-ring of $R$ modulo $A$. Denote by $\rho_A: R\longrightarrow R/A$ the canonical projection sending $\lambda\in R$ to $\bar{\lambda}=\lambda+A\in R/A$. Applying the projection to all entries of a matrix, we get the reduction homomorphism 
$$\begin{array}{rcl}
\rho_{A}:\bw{2}\GL_{n}(R)&\longrightarrow& \bw{2}\GL_{n}(R/A)\\
a &\mapsto& \overline{a}=(\overline{a}_{I,J})
\end{array}$$
The kernel of the homomorphism $\rho_{A}$, $\bw{2}\GL_{n}(R,A)$, is called the \textit{principal congruence subgroup} of level $A$, whereas the \textit{full congruence subgroup} $\CC\bw{2}\GL_{n}(R,A)$ is the full preimage of the center of $\bw{2}\GL_{n}(R/A)$ under $\rho_{A}$. Let us remark that $\bw{2}\GL_{n}(R,A)\leq\CC\bw{2}\GL_{n}(R,A)$ and both groups $\bw{2}\GL_{n}(R,A)$ and $\CC\bw{2}\GL_{n}(R,A)$ are normal in $\bw{2}\GL_{n}(R)$.
\begin{prop}\label{scf}
Let $R$ be a commutative ring, $n\geq 3$. Then for any ideal $A\trianglelefteq R$ the equality
$$[\CC\bw{2}\GL_n(R,A),\bw{2}\E_n(R)]=\bw{2}\E_n(R,A)$$
holds
\end{prop}
This result is called the \textit{standard commutator formula}. It has a lot of proofs in various contexts, see e.\,g., the paper~\cite{HVZrelachev}.

The \textit{upper level} of a matrix $g\in\bw{2}\GL_n(R)$ is the smallest ideal $I=\level(g)\trianglelefteq R$ such that $g\in \CC\bw{2}\GL_n(R,I)$. As in the case of the general linear group, the upper level is generated by the off-diagonal entries $g_{I,J}$, $I\neq J$ and by the pair-wise differences of its diagonal entries $g_{I,I}-g_{J,J}$, $I\neq J$. Note that it suffices to consider only the ``fundamental'' differences $g_{I,I}-g_{\tilde{I},\tilde{I}}$, where $\tilde{I}$ is the next index after $I$ in $\bw{2}[n]$ in ascending order. Thus, the upper level $\level(g)$ is generated by $\binom{n}{2}^2-1$ elements.

\begin{lemma}\label{monomial}
Let $P_{ij}:=\bw{2}t_{i,j}(1)\bw{2}t_{j,i}(-1)\bw{2}t_{i,j}(1)\in\bw{2}\E_n(R)$, where $1\leq i\neq j\leq n$.
Then for any index $k\neq i,j$ and any $\xi\in R$:
\begin{itemize}
    \item ${}^{P_{ki}}\bw{2}t_{i,j}(\xi)=\bw{2}t_{k,j}(\xi)$;
    \item ${}^{P_{kj}}\bw{2}t_{i,j}(\xi)=\bw{2}t_{i,k}(\xi)$.
\end{itemize}
\end{lemma}
The proof is straightforward.

To formulate the main result of the present paper we introduce some notation. A matrix of the form $g^{\pm h}$ is called an elementary [exterior] $g$-conjugate, where $g\in \bw{2}\GL_n(R)$ and $h\in\bw{2}\E_n(R)$.

\setcounter{theorem}{0}
\begin{theorem}\label{RDT}
Let R be a commutative ring, $n\geq 4$, and $g\in\bw{2}\GL_n(R)$. Then for any $\xi\in \level(g)$ the transvection $\bw{2}t_{k,l}(\xi)$ is a product of $\leq 8\big(\binom{n}{2}^2-1\big)$ elementary conjugates of $g$ and $g^{-1}$. Namely,
\begin{enumerate}
    \item\label{p8} $\bw{2}t_{k,l}(g_{I,J})$ is a product of eight elementary exterior $g$-conjugates, where $\height(I,J)=1$;
    \item\label{p16} $\bw{2}t_{k,l}(g_{I,J})$ is a product of sixteen elementary exterior $g$-conjugates, where $\height(I,J)=0$;
    \item\label{p24} $\bw{2}t_{k,l}(g_{I,I}-g_{J,J})$ is a product of $24$ elementary exterior $g$-conjugates, where $\height(I,J)=1$;
    \item\label{p48} $\bw{2}t_{k,l}(g_{I,I}-g_{J,J})$ is a product of $48$ elementary exterior $g$-conjugates, where $\height(I,J)=0$.
\end{enumerate}
\end{theorem}

Before we prove the theorem, we formulate a corollary. The latter result can be regarded as a stronger version of the standard description of $\bw{2}\E_n(R)$--normalized subgroups. Hence, we have one more very
short proof of the Sandwich Classification Theorem for the exterior square of elementary groups.

\setcounter{theorem}{3}
\begin{theorem}
Let $H$ be a subgroup of $\bw{2}\GL_n(R)$. Then $H$ is normalized by $\bw{2}\E_n(R)$ if and only if
$$\bw{2}\E_n(R,A)\leq H\leq\CC\bw{2}\GL_n(R,A)$$
for some ideal $A$ of $R$.
\end{theorem}
\begin{proof}
Let $H$ is normalized by $\bw{2}\E_n(R)$. Set $A:=\{\xi\in R\;|\; \bw{2}t_{1,2}(\xi)\in H\}$. It is obvious that $\bw{2}\E_n(R,A)\leq H$. It remains to check that if $g\in H$ then $g_{I,J}$, $g_{I,I}-g_{\tilde{I},\tilde{I}}\in A$. But Theorem~\ref{RDT} states exactly this claim. Hence, $H\leq\CC\bw{2}\GL_n(R,A)$. The converse follows from Proposition~\ref{scf}.
\end{proof}

Now we must only check Theorem~\ref{RDT}. A proof of statement~$(\ref{p8})$ is a key step in our verification. Other cases follow from the fist one.

\begin{proof} 
We will use the first method of a column stabilization (without the Pl\"ucker relations) to cover the case $n=4$. Let $T=T_{*,1}=\prod\limits_{s\neq 1}\bw{2}t_{s,1}(g_{1s,12})$. By Proposition~\ref{stabil} the first column of $Tg$ equals the first column of the matrix $g$. Hence the first column of the matrix $h:=g^{-1}Tg$ is standard, i.\,e., $h$ lies in the parabolic subgroup $P_{12}$. By Lemma \ref{equat} $h$ also lies in the submaximal parabolic subgroup ${}_KP_{12}$, where $K$ corresponds the latter $\binom{n-2}{2}$ rows.

Next, note that for any $j\in \{3,4,5,\dots,n\}$ the exterior transvections $\bw{2}t_{1,j}(\xi)$ and $\bw{2}t_{2,j}(\xi)$ sit in the unipotent radical $U$ of the parabolic subgroup ${}_KP_{12}$. Using obvious formula $[xy,z]^{x}=[y,z]\cdot [z,x^{-1}]$, we get
$$z:=[T^{-1}h,\bw{2}t_{2,3}(1)]^{T^{-1}}=[h,\bw{2}t_{2,3}(1)]\cdot [\bw{2}t_{2,3}(1),T].$$
Now the matrix $z$ is a product of four elementary exterior conjugates of $g$ and $g^{-1}$. The first commutator $[h,\bw{2}t_{2,3}(1)]$ belongs to the unipotent radical $U$, whereas the second one equals $\bw{2}t_{2,1}(g_{13,12})$.
Since the transvection $\bw{2}t_{1,3}(1)$ also sits in $U$ and the unipotent radical is abelian, we obtain
$$[\bw{2}t_{1,3}(-1),z]=[\bw{2}t_{1,3}(-1),u\cdot \bw{2}t_{2,1}(g_{13,12})]=\bw{2}t_{2,3}(g_{13,12}).$$
Consequently the transvection $\bw{2}t_{2,3}(g_{13,12})$ is a product of eight elementary exterior conjugates of $g$ and $g^{-1}$. By Lemma~\ref{monomial} it follows that $\bw{2}t_{k,l}(g_{13,12})$ is a product of eight elementary exterior $g$-conjugates. It remains to note that we can bring $g_{I,J}$ to position $(13,12)$ conjugating by monomial matrices from $\bw{2}\E_n(R)$. 

Since $n\geq 4$ there are two distinct indices $h_1, h_2\in [n]\setminus \{i,j\}$. Let us remark that the entry of $g^{\bw{2}t_{i,j}(-1)}$ in the position $(ih_1,ih_2)$ equals $g_{jh_1,ih_2}+g_{ih_1,ih_2}$. Applying (\ref{p8}) to $g^{\bw{2}t_{i,j}(-1)}$, we get that $\bw{2}t_{k,l}(g_{jh_1,ih_2}+g_{ih_1,ih_2})$ is a product of eight elementary exterior $g$-conjugates. Therefore,
$$\bw{2}t_{k,l}(g_{jh_1,ih_2})=\bw{2}t_{k,l}(g_{jh_1,ih_2}+g_{ih_1,ih_2})\bw{2}t_{k,l}(-g_{ih_1,ih_2})$$
is a product of sixteen elementary exterior $g$-conjugates.

The third assertion follows from the above. Obviously, the entry of $g^{\bw{2}t_{i,j}(1)}$ in the position $(ih,jh)$ equals $g_{ih,ih}-g_{jh,jh}+g_{ih,jh}-g_{jh,ih}$ for any index $h\neq i,j$. Applying (\ref{p8}), we obtain that $\bw{2}t_{k,l}(g_{ih,ih}-g_{jh,jh}+g_{ih,jh}-g_{jh,ih})\in\bw{2}\E_n(R)$ is a product of eight elementary exterior $g$-conjugates. Finally,
$$\bw{2}t_{k,l}(g_{ih,ih}-g_{jh,jh})=\bw{2}t_{k,l}(g_{ih,ih}-g_{jh,jh}+g_{ih,jh}-g_{jh,ih})\bw{2}t_{k,l}(g_{jh,ih}-g_{ih,jh})$$
is a product of 24 elementary exterior $g$-conjugates.

To finish the proof of the theorem it only remains to check the last assertion. There is an index $K\in\bw{2}[n]$ such that $\height(I,K)=\height(J,K)=1$. Therefore,
$$\bw{2}t_{k,l}(g_{I,I}-g_{J,J})=\bw{2}t_{k,l}(g_{I,I}-g_{K,K})\bw{2}t_{k,l}(g_{K,K}-g_{J,J})$$
is a product of 48 elementary exterior $g$-conjugates.
\end{proof}

\bibliographystyle{zapiski}
\bibliography{english}

\end{document}